\documentclass[12pt]{amsart}
\usepackage{amsmath,amssymb,amsfonts,amsthm}
\usepackage{epsfig}
\usepackage{a4wide}
\usepackage{enumerate,color}
\usepackage{subfigure,hyperref}
\newtheorem{theorem}{Theorem}[section]
\newtheorem{lemma}[theorem]{Lemma}
\newtheorem{corollary}[theorem]{Corollary}
\newtheorem{prob}[theorem]{Problem}

\newtheorem{claim}[theorem]{Claim}
\newtheorem{conjecture}[theorem]{Conjecture}

\theoremstyle{definition}

\def\cD{\mathcal{D}}

\def\cF{\mathcal{F}}

\def\cS{\mathcal{S}}

\def\eps{\varepsilon}

\let\oldrceil\rceil
\renewcommand{\rceil}{\right\oldrceil}
\let\oldlceil\lceil
\renewcommand{\lceil}{\left\oldlceil}

\begin{document}

\title{Coloring Jordan regions and curves}

\author[W. Cames van Batenburg]{Wouter Cames van Batenburg} 
\address{Department of Mathematics, Radboud University Nijmegen, The Netherlands}
\email{w.camesvanbatenburg@math.ru.nl}

\author[L. Esperet]{Louis Esperet} 
\address{Univ. Grenoble Alpes, CNRS, G-SCOP, Grenoble, France}
\email{louis.esperet@grenoble-inp.fr}

\author[T. M\"uller]{Tobias M\"uller} 
\address{Utrecht University, Utrecht, The Netherlands}
\email{t.muller@uu.nl}
\thanks{Louis Esperet and Tobias M\"uller were partially supported by ANR Project Stint
  (\textsc{anr-13-bs02-0007}) and LabEx PERSYVAL-Lab
  (\textsc{anr-11-labx-0025-01}). Wouter Cames van Batenburg was
  supported by NWO (\textsc{613.001.217}) and a bilateral project PHC
  Van Gogh 2016 (35513NM). Tobias M\"uller's research was
partially supported by an NWO VIDI grant}

\date{}

\sloppy

\begin{abstract}
A Jordan region is a subset of the plane that is homeomorphic to a
closed disk. Consider a family $\cF$ of Jordan regions whose interiors are
pairwise disjoint, and such that any two Jordan regions intersect in at
most one point. If any point of the plane is contained in at most
$k$
elements of $\cF$ (with $k$ sufficiently large), then we show that the elements of $\cF$ can be
colored with at most $k+1$ colors so that intersecting Jordan regions
are assigned distinct colors. This is best possible and answers a question raised by Reed and Shepherd in
1996. As a simple corollary, we also obtain a positive answer to a
problem of Hlin\v en\'y (1998) on the chromatic
number of contact systems of strings. 

We also investigate the
chromatic number of families of touching Jordan curves. This can be used to bound the ratio
between the maximum number of vertex-disjoint
directed cycles in a planar digraph, and its fractional counterpart.
\end{abstract}

\maketitle

\section{Introduction}

In this paper, a \emph{Jordan region} is a subset of the plane that is homeomorphic to a
closed disk. 
A family $\cF$ of Jordan regions is \emph{touching} if their interiors are pairwise disjoint. If any point of the
plane is contained in at most $k$ Jordan regions of $\cF$, then we say
that $\cF$ is \emph{$k$-touching}. If any two elements of
$\cF$ intersect in at most one point, then $\cF$ is said to be
\emph{simple}. All the families of Jordan regions and curves we consider in this
paper are assumed to have a finite number of intersection points. The first part of this paper is concerned with the \emph{chromatic number} of simple
$k$-touching families of Jordan regions, i.e. the minimum number of
colors needed to color the Jordan regions, so that intersecting
Jordan regions receive different colors. This can also be defined as
the chromatic number of the \emph{intersection graph} $G(\cF)$ of $\cF$, which is the graph with vertex set $\cF$ in which two vertices
are adjacent if and only if the corresponding elements of $\cF$
intersect. Recall that the \emph{chromatic number} of a graph $G$,
denoted by $\chi(G)$, is the least number of colors needed to color
the vertices of $G$, so that adjacent vertices receive different
colors. The chromatic number of a graph $G$ is at least the
\emph{clique number} of $G$, denoted by $\omega(G)$, which is the
maximum number of pairwise adjacent vertices in $G$, but the
difference between the two parameters can be arbitrarily large
(see~\cite{Kos04} for a survey on the chromatic and clique numbers of geometric
intersection graphs).

The following question was raised by Reed and Shepherd~\cite{RS96}.

\begin{prob}{\cite{RS96}}\label{prob:1}
Is there a constant $C$ such that for any simple touching family $\cF$
of Jordan regions, $\chi(G(\cF))\le \omega(G(\cF))+C$? Can we take $C=1$?
\end{prob}

Our main result is the following (we made no real effort to optimize the
constant 490, which is certainly far from optimal, our main concern
was to give a proof that is as simple as possible).

\begin{theorem}\label{th:k1}
For $k\ge 490$, any simple $k$-touching family of Jordan regions is $(k+1)$-colorable.
\end{theorem}

Note that apart from the constant 490, Theorem~\ref{th:k1} is best
possible. Figure~\ref{fig:clique} depicts two examples of simple
$k$-touching families of Jordan regions of chromatic number $k+1$. 

\begin{figure}[htbp]
\begin{center}
\includegraphics[scale=1.2]{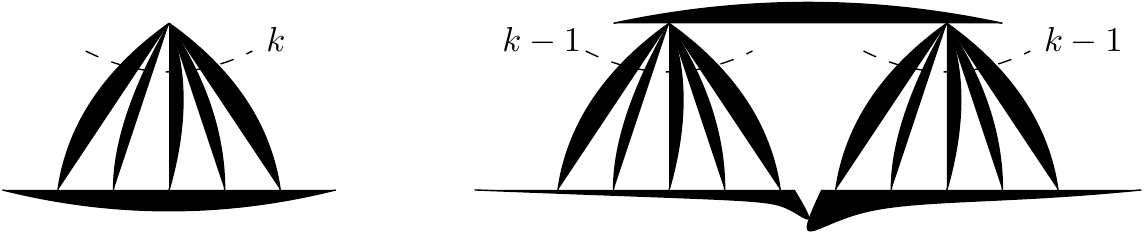}
\caption{Two simple
$k$-touching families of Jordan regions with chromatic number $k+1$. \label{fig:clique}}
\end{center}
\end{figure}

It was proved in~\cite{EGL16} that every simple $k$-touching family of
Jordan regions is $3k$-colorable (their result is actually stated for
$k$-touching families of strings, but it easily implies the
result on Jordan regions). We obtain the next result as a simple consequence. 

\begin{corollary}\label{cor:1}
Any simple $k$-touching family of Jordan regions is $(k+327)$-colorable.
\end{corollary}

\begin{proof}
Let $\cF$ be a simple $k$-touching family $\cF$ of Jordan regions. If
$k\le 163$ then $\cF$ can be colored with at most $3k\le k+327$ colors
by the result of~\cite{EGL16} mentioned above. If $164\le k \le 489$, then $\cF$ is also
$490$-touching, and it follows from Theorem~\ref{th:k1} that $\cF$ can
be colored with at most $491\le k+327$ colors.
Finally, if $k\ge 490$, Theorem~\ref{th:k1} implies that $\cF$ can
be colored with at most $k+1\le k+327$ colors.
\end{proof}

Observe that for a given simple touching family $\cF$ of Jordan regions,
if we denote by $k$ the least integer so that $\cF$ is $k$-touching,
then $\omega(G(\cF))\ge k$, since $k$ Jordan regions intersecting some
point $p$ 
of the plane are pairwise intersecting. Therefore, we obtain the
following immediate corollary, which is a positive
answer to the problem raised by Reed and Shepherd.

\begin{corollary}\label{cor:2}
For any simple touching family $\cF$ of Jordan regions, $\chi(G(\cF))\le
\omega(G(\cF))+327$ (and $\chi(G(\cF))\le
\omega(G(\cF))+1$ if $\omega(G(\cF))\ge 490$).
\end{corollary}

Note that the bound $\chi(G(\cF))\le
\omega(G(\cF))+1$ is also best possible (as shown by
Figure~\ref{fig:clique}, right).

\smallskip

It turns out that our main result also implies a positive answer to a
question raised by Hlin\v en\'y in 1998~\cite{Hli98}. A \emph{string}
is the image of some continuous injective function from $[0,1]$ to
$\mathbb{R}^2$, and the \emph{interior} of a string is the string
minus its two endpoints. A \emph{contact systems of strings} is a set of strings such
that the interiors of any two strings have empty intersection. In other
words, if $c$ is a contact point in the interior of a string $s$, all
the strings containing $c$ distinct from $s$ end at
$c$. A contact system of strings is said to be \emph{one-sided} if for
any contact point $c$ as above, all the strings ending at $c$ leave from the same side of
$s$ (see Figure~\ref{fig:hli}, left). Hlin\v en\'y~\cite{Hli98} raised the following problem:

\begin{prob}{\cite{Hli98}}\label{prob:2}
Let $\cS$ be a one-sided contact system of strings, such that any
point of the plane is in at most $k$ strings, and any two strings
intersect in at most one point. Is it true that $G(\cS)$ has chromatic
number at most $k+o(k)$? (or even $k+c$, for some constant $c$?)
\end{prob}

\begin{figure}[htbp]
\begin{center}
\includegraphics[scale=1]{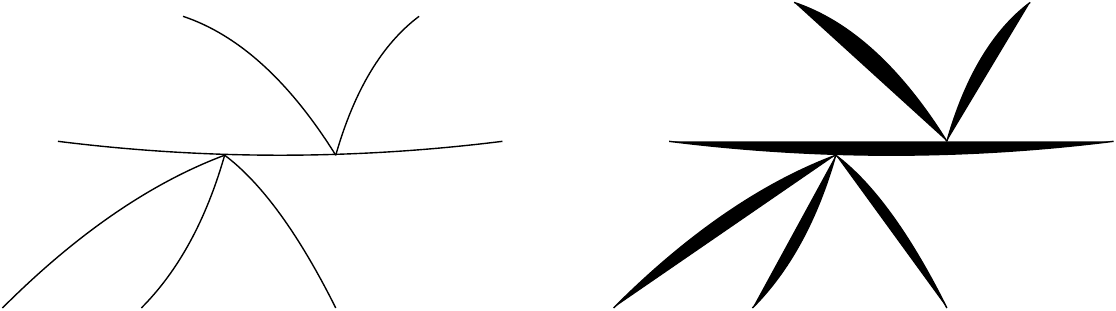}
\caption{Turning a one-sided contact system of strings into a simple
  touching family of Jordan regions. \label{fig:hli}}
\end{center}
\end{figure}

The following simple corollary of Theorem~\ref{th:k1} gives a positive answer to
Problem~\ref{prob:2}.

\begin{corollary}\label{cor:3}
Let $\cS$ be a one-sided contact system of strings, such that any
point of the plane is in at most $k$ strings, and any two strings
intersect in at most one point. Then $G(\cS)$ has chromatic number at most $k+127$ (and at most $k+1$ if $k\ge 490$).
\end{corollary}

\begin{proof}
Assume first that $k\le 363$. It was proved in~\cite{EGL16} that $G(\cS)$ has chromatic number at most $\lceil \frac{4}{3} k \rceil + 6$, so in this case at most
$k+127$, as desired. Assume now that $k\ge 364$.
Let $\cF$ be obtained from $\cS$ by thickening each string $s$ of $\cS$, turning $s$ into a (very
thin) Jordan region (see Figure~\ref{fig:hli}, from left to right). Since $\cS$ is one-sided, each intersection point
contains precisely the same elements in $\cS$ and $\cF$, and therefore
$G(\cS)$ and $G(\cF)$ are equal, while $\cF$ is a simple $k$-touching
family of Jordan regions. If $364\le k\le 489$, then $\cF$ is also
$490$-touching and it follows from  Theorem \ref{th:k1} that $G(\cS)=G(\cF)$ has
chromatic number at most $491\le k+127$. Finally, if $k\ge 490$, then by Theorem \ref{th:k1}, $G(\cS)=G(\cF)$ has
chromatic number at most $k+1$, as desired.
\end{proof}

A \emph{Jordan curve} is the boundary of some Jordan region of the
plane. We say that a family of Jordan curves is \emph{touching} if
for any two Jordan curves $a,b$, the curves $a$ and $b$ do not cross
(equivalently, either the interiors of the regions
bounded by $a$ and $b$ are disjoint, or one is contained in the
other). Moreover, if any point of the plane is on at most $k$
Jordan curves, we say that the family is $k$-touching. Note that
unlike above, the
families of Jordan curves we consider here are not required to be
simple (two Jordan curves may intersect in several points). Note that previous
works on intersection of Jordan curves have usually considered the
opposite case, where every two curves that intersect also cross (see
for instance~\cite{KM13} and the references therein).

Let $\cF$ be a $k$-touching family of Jordan curves.
For any two
intersecting Jordan curves $a,b\in \cF$, let $\cD(a,b)$ be the set of
Jordan curves $c$ distinct
from $a,b$ such that the (closed) region bounded by $c$ contains
exactly one of $a,b$.
The cardinality of
$\cD(a,b)$ is called \emph{the distance between} $a$ and
$b$, and is denoted by $d(a,b)$. Note that since $\cF$ is $k$-touching, any two
intersecting Jordan curves are at distance at most $k-2$. 

Given a $k$-touching family $\cF$, the \emph{average distance in}
$\cF$ is the average of $d(a,b)$, over all pairs of intersecting Jordan curves
$a,b\in \cF$. We conjecture the following.

\begin{conjecture}\label{conj:ad}
For any $k$-touching family $\cF$ of
Jordan curves, the average distance in $\cF$ is at most $\tfrac{k}2$.
\end{conjecture}

It was proved by Fox and Pach~\cite{FP10} that each $k$-touching
family of strings is $(6ek+1)$-colorable, which directly implies
that each $k$-touching family of Jordan curves is
$(6ek+1)$-colorable (note that
$6e\approx 16.31$). We show how to improve this bound when the average
distance is at most $\alpha k$, for some $\alpha\le 1$.\\

\begin{theorem}~\label{th:ad}
Let $\cF$ be a $k$-touching family of Jordan curves, such that the
average distance in $\cF$ is at most $\alpha k$, for some constant
$0\le\alpha \le 1$. Then the chromatic
number of $\cF$ is at most
$\tfrac{6e^\delta}{\delta+\delta^2(1-\alpha)}\,k$, where
$\delta=\delta(\alpha)=\tfrac1{2-2\alpha}(1-2\alpha+\sqrt{4
  \alpha^2-8\alpha+5})$ for $\alpha<1$ and $\delta(1)=1$.
\end{theorem}

Note that $\delta(1)=1=\lim_{\alpha\rightarrow 1}\tfrac1{2-2\alpha}(1-2\alpha+\sqrt{4
  \alpha^2-8\alpha+5})$.
Theorem~\ref{th:ad} has the following direct corollary.

\begin{corollary}\label{cor:ad}
Let $\cF$ be a $k$-touching family of Jordan curves, such that the
average distance in $\cF$ is at most $\alpha k$. Then $\cF$ is
$\beta\,k$-colorable, where 

$$
\beta= \left\{
    \begin{array}{lll}
        12.76 & \mbox{if} & \alpha\le
3/4, \\
        10.22 & \mbox{if} & \alpha\le
1/2, \\
8.43 & \mbox{if} & \alpha\le
1/4.
    \end{array}
\right.
$$
\end{corollary}

By Corollary~\ref{cor:ad}, a direct consequence of Conjecture~\ref{conj:ad} would be that
every $k$-touching
family of Jordan curves is $10.22\,k$-colorable.

\smallskip

For any $k$-touching family of Jordan curves, the average distance is at most
$k$. Theorem~\ref{th:ad} implies that every family of Jordan curves
is $6ek$-colorable, which is the bound of Fox
and Pach~\cite{FP10} (without the $+1$). 
To understand the limitation of Theorem~\ref{th:ad} it is interesting 
to consider the case $\alpha=o(1)$. Then $\delta$ tends to
$\tfrac12(1+\sqrt{5})$, and we obtain in this case that $\cF$ is $7.14\, k$-colorable. A
particular case is when $\alpha=0$. This is equivalent to say that any
two intersecting Jordan curves are at distance 0, and therefore the
family $\cF$ of Jordan curves can be turned into a $k$-touching family of
Jordan regions (here and everywhere else in this manuscript, it is
crucial that the curves are pairwise non-crossing). Note that it was proved
in~\cite{AEH13} (see also~\cite{EGL16}) that $k$-touching families of
Jordan regions are $(\tfrac{3k}2+o(k))$-colorable.

\medskip

In order to motivate Conjecture~\ref{conj:ad} and give it some credit,
we then prove the following weaker version.

\begin{theorem}\label{th:avd2}
Let $\cF$ be a family of $k$-touching Jordan curves. Then the average
distance in $\cF$ is at most $k/(1+\tfrac1{16e})$.
\end{theorem}

An immediate consequence of Theorems~\ref{th:ad} and~\ref{th:avd2} is
the following small improvement over the bound of Fox
and Pach~\cite{FP10} in the case of Jordan curves. 

\begin{corollary}\label{cor:ad2}
Any $k$-touching family of Jordan curves is $15.95k$-colorable.
\end{corollary}

An
interesting connection between the chromatic number of $k$-touching
families of Jordan curves and the packing number of directed cycles in
directed planar graphs was observed by Reed and Shepherd in
\cite{RS96}. In a planar digraph $G$, let $\nu(G)$ be the maximum
number of vertex-disjoint directed cycles. This quantity has a natural
linear relaxation, where we seek the maximum $\nu^*(G)$ for which
there are weights in $[0,1]$ on each directed cycle of $G$, summing up
to $\nu^*(G)$, such that for each vertex $v$ of $G$, the sum of the
weights of the directed cycles containing $v$ is at most 1. It was
observed by Reed and Shepherd~\cite{RS96} that for any $G$ there are integers
$n$ and $k$ such that $\nu^*(G)=\tfrac{n}{k}$ and $G$ contains a collection
of $n$ pairwise non-crossing directed cycles (counted with
multiplicities) such that each vertex is in at most $k$ of the
directed cycles. If we replace each directed cycle of the collection by its
image in the plane, we obtain a $k$-touching family of
Jordan curves. Assume that this family is
$\beta\,k$-colorable, for some constant $\beta$. Then the family contains an independent set (a
set of pairwise non-intersecting Jordan curves) of size at least
$n/(\beta\,k)$. This independent set corresponds to a packing of
directed cycles in $G$. As a consequence, $\nu(G)\ge
n/(\beta\,k)=\nu^*(G)/\beta$, and then $\nu^*(G)\le
\beta\, \nu(G)$. The following is therefore a direct
consequence of Corollaries~\ref{cor:ad} and~\ref{cor:ad2}.

\begin{theorem}\label{th:gap}
For any planar directed graph $G$, $\nu^*(G)\le 15.95\cdot
\nu(G)$. Moreover, if Conjecture~\ref{conj:ad} holds, then $\nu^*(G)\le 10.22\cdot
\nu(G)$
\end{theorem}

This improves a result of Reed and Shepherd~\cite{RS96}, who proved that for any planar directed graph $G$, $\nu^*(G)\le 28\cdot
\nu(G)$. The same result with a constant factor of $16.31$ essentially
followed from the result of Fox and Pach~\cite{FP10} (and
the discussion above). Using classical results of Goemans and Williamson~\cite{GM97},
Theorem~\ref{th:gap} also gives improved bounds on the ratio between
the maximum packing of directed cycles in planar digraphs and the dual version
of the problem, namely the minimum number of vertices that needs to be
removed from a planar digraph in order to obtain an acyclic digraph.

\medskip

\noindent {\bf Organization of the paper.} The proofs of
Theorem~\ref{th:k1}, \ref{th:ad} and~\ref{th:avd2} are given in
Sections~\ref{sec:k1}, \ref{sec:ad} and~\ref{sec:avdist},
respectively. Section~\ref{sec:ccl} concludes the paper with some
remarks and open problems.

\section{Proof of Theorem~\ref{th:k1}}\label{sec:k1}

In the proof below we will use the following parameters instead of their
numerical values (for the sake of readability): $\eps=\tfrac14$, $b=\tfrac{18}{\eps}=72$, and $k\ge
7b-14=490$.

\smallskip

The proof proceeds by contradiction. Assume that there exists a
counterexample $\cF$, and take it with a minimum number of
Jordan regions.

We will construct a bipartite planar graph $G$ from $\cF$ as follows: for any
Jordan region $d$ of $\cF$ we add a vertex in the interior of $d$ (such
a vertex will be called a \emph{disk vertex}), and for any contact
point $p$ (i.e.\ any point on at least two Jordan regions), we add a
new vertex at $p$ (such a vertex will be called a \emph{contact
  vertex}). Now, for every Jordan region $d$ and contact point $p$
on $d$, we add an edge between the disk vertex corresponding
to $d$ and the contact vertex corresponding to $p$.

We now start with some remarks on the structure of $G$.

\begin{claim}\label{cl:cbp}
$G$ is a connected bipartite planar graph.
\end{claim}

\begin{proof}
The fact that $G$ is planar and bipartite easily follows from the
construction. If $G$ is disconnected, then $G(\cF)$ itself is
disconnected, and some connected component contradicts the minimality
of $\cF$.
\end{proof}

\begin{claim}\label{cl:g6}
All the faces of $G$ have degree (number of edges in a boundary walk
counted with multiplicity) at least 6.
\end{claim}

\begin{proof}
Note that by construction, the graph $G$ is simple (i.e.\ there are no
parallel edges). Assume for the sake of contradiction that $G$ has a
face $f$ of degree 4. Then either $f$ bounds three vertices (and $\cF$
consists of two Jordan curves intersecting in a single point, in which
case the theorem trivially holds), or the face $f$ corresponds to two Jordan regions of
$\cF$ sharing two distinct points, which contradicts the fact that
$\cF$ is simple. Since $G$ is bipartite, it follows that each face has degree at least 6.
\end{proof}

Two disk vertices having a common neighbor are said to be \emph{loose
  neighbors} in $G$ (this corresponds to intersecting Jordan regions in $\cF$).

\begin{claim}\label{cl:k}
Every disk vertex has at least $k+1$ loose neighbors in $G$.
\end{claim}

\begin{proof}
Assume that some disk vertex has at most $k$ loose neighbors in
$G$. Then the corresponding Jordan region $d$ of $\cF$ intersects at
most $k$ other Jordan regions in $\cF$. By minimality of $\cF$, the family
$\cF\setminus \{d\}$ is $(k+1)$-colorable, and any $(k+1)$-coloring
easily extends to $d$, since $d$ intersects at most $k$ other
Jordan regions. We obtain a $(k+1)$-coloring of $\cF$, which is a contradiction.
\end{proof}

\begin{claim}\label{cl:min2}
$G$ has minimum degree at least 2, and each contact vertex has degree at most $k$.
\end{claim}

\begin{proof}
The fact that each contact vertex has degree at least two and at most $k$ directly
follows from the definition of a $k$-touching family. If $G$ contains
a disk vertex $v$ of degree at most one, then since contact vertices
have degree at most $k$, $v$ has at most $k-1$ loose neighbors in $G$,
which contradicts Claim~\ref{cl:k}.
\end{proof}

\begin{claim}\label{cl:22}
For any edge $uv$, at least one of $u,v$ has degree at least 3.
\end{claim}

\begin{proof}
Assume that a disk vertex $u$ of degree 2 is adjacent to a contact
vertex of degree 2. Then $u$ has at most $1+k-1=k$ loose
neighbors, which contradicts Claim~\ref{cl:k}. 
\end{proof}

A $d$-vertex (resp. $\le d$-vertex, $\ge d$-vertex) is a vertex of degree $d$ (resp. at most $d$, at least
$d$). A $\ge b$-vertex is also said to be a \emph{big vertex}. A
vertex that is not big is said to be \emph{small}.

\begin{claim}\label{cl:big}
Each disk vertex of degree at most 7 has at least one big neighbor.
\end{claim}

\begin{proof}
Assume that some disk vertex $v$ of degree at most 7 has no big
neighbor. It follows that all the neighbors of $v$ have degree at most
$b-1$, and so $v$ has at most $7(b-2)\le k$ loose neighbors, which
contradicts Claim~\ref{cl:k}.
\end{proof}

We now assign to each vertex $v$ of $G$ a charge $\omega(v)=2d(v)-6$,
and to each face $f$ of $G$ a charge $\omega(f)=d(f)-6$ (here the
function $d$ refers to the degree of a vertex or a face). By Euler's
formula, the total charge assigned to the vertices and edges of $G$ is
precisely $-12$. We now proceed by locally moving the charges
(while preserving the total charge) until all vertices and faces have
nonnegative charge. In this case we obtain that $-12\ge 0$, which is a
contradiction. The charges are locally redistributed according to the
following rules (for Rule (R2), we need the following definition: a
\emph{bad vertex} is a disk 3-vertex $v$ adjacent to two contact
2-vertices $u,w$, such that the three faces incident to $v$ have
degree 6 and the neighbors of $u$ and $w$ have degree 3).

\smallskip

\begin{enumerate}[(R1)]

\item For each big contact vertex $v$ and each sequence of three
  consecutive neighbors $u_1,u_2,u_3$ of $v$ in clockwise order around
  $v$, we do the following. If $u_2$ has a unique big neighbor
  (namely, $v$),
  then $v$ gives $2-\eps$ to $u_2$. Otherwise $v$ gives 1 to $u_2$,
  and $(1-\eps)/2$ to each of $u_1$ and $u_3$.

\item Each big contact vertex gives $\eps$ to each bad neighbor. 

\item Each small contact vertex of degree at least 4 gives $\tfrac12$
  to each neighbor.

\item Each contact 3-vertex adjacent to some $\ge 3$-vertex gives
  $\eps$ to each neighbor of degree 2.

\item Each disk vertex of degree at least 4 gives $1+\eps$ to each
  neighbor of degree at most 3.

\item For each disk vertex $v$ of degree 3 and each neighbor $u$ of
  $v$ with $d(u)\le 3$, we do the following. If either $u$ has degree
  3, or $u$ has degree two and the neighbor of $u$ distinct from $v$
  has degree at least 4, then $v$ gives $1-\eps$ to $u$. Otherwise,
  $v$ gives 1 to $u$.

\item Each face $f$ of degree at least 8 gives $\tfrac12$ to each disk
  vertex incident with $f$.

\end{enumerate}

We now analyze the new charge of each vertex and face after all these
rules have been applied. 

By Claim~\ref{cl:g6}, all faces have degree at
least 6. Since faces of degree 6 start with a charge of 0, and do not
give any charge, their new charge is still 0. Let $f$ be a face of
degree $d\ge 8$. Then $f$ starts with a charge of $d-6$ and gives at
most $\tfrac{d}2\cdot \tfrac12$ by Rule (R7). The new charge is then
at least $d-6-\tfrac{d}2\cdot \tfrac12=\tfrac{3d}4-6\ge 0$, as
desired.

\smallskip

We now consider disk vertices. Note that these vertices receive charge
by Rules (R1--4) and (R7), and give charge by Rules (R5--6). Consider
first a disk vertex $v$ of degree $d\ge 8$. Then $v$ starts with a charge
of $2d-6$ and gives at most $d(1+\eps)$ (by Rule (R5)), so the new
charge of $v$ is at least $2d-6-d(1+\eps)=d(1-\eps)-6\ge 0$ (since $\eps= \tfrac14$). 

Assume now that $v$ is a disk vertex of degree $4\le d \le 7$. Then
by Claim~\ref{cl:big}, $v$ has at least one big neighbor. The vertex $v$ starts with a charge of $2d-6$, receives at least $2-\eps$ by Rule
(R1), and gives at most
$(d-1)(1+\eps)$ by Rule (R5). The new charge of $v$ is then at least
$2d-6+2-\eps -(d-1)(1+\eps)=d(1-\eps)-3\ge 0$ (since $\eps= \tfrac14$).

We now consider a disk vertex $v$ of degree 3. Again, it follows
from Claim~\ref{cl:big} that $v$ has at least one big neighbor. The
vertex $v$ starts with a charge of 0, and since $v$ has at least one
big neighbor, $v$ receives at least $2-\eps$
from its big neighbors by Rule
(R1). Let $w$ be a big neighbor of $v$, and assume first that at least one of the two neighbors of $v$
distinct from $w$ (call them $u_1,u_2$) is not a 2-vertex adjacent to
two 3-vertices. Then by Rule (R6), $v$
gives at most $2-\eps$ to $u_1,u_2$ (recall that by Claim~\ref{cl:22}, no two vertices of degree $2$ are adjacent in $G$). In this case the new
charge of $v$ is at least $2-\eps-(2-\eps)\ge 0$, as desired. Assume
now that $u_1,u_2$ both have degree two and their
neighbors all have degree 3. In this case $v$ gives 1 to each of
$u_1,u_2$ and the new charge of $v$ is at least $2-\eps-2\ge
-\eps$. If $v$ is incident to a face of degree at least 8, $v$
receives at least $\tfrac12$ from such a face, and its new charge is
at least $-\eps+\tfrac12\ge 0$, as desired. So we can assume that all
the faces incident to $v$ are faces of degree 6. In other words, $v$
is a bad vertex. Then $w$ gives an additional charge of $\eps$ to $v$
by Rule (R2), and the new charge of $v$ in
this last case is at least $-\eps+\eps\ge 0$, as desired.

Assume now that $v$ is a disk vertex of degree two. Then the vertex $v$ starts with a charge of $-2$.
By
Claim~\ref{cl:big}, $v$ has a big neighbor, call it $w$. By
Claim~\ref{cl:22}, the neighbor of $v$ distinct from $w$, call it $u$,
has degree at least 3. If $u$ is big then $v$
receives a charge of $1+1=2$ by Rule (R1) and its new charge is thus at
least $-2+2=0$, so we can assume that $u$ is small (in particular, $v$
receives $2-\eps$ from $w$ by Rule (R1)). If $u$ has degree at
least 4, then $u$ gives a charge of $\tfrac12$ to $v$ by Rule (R3) and the new
charge of $v$ is then at least $-2+2-\eps+\tfrac12\ge 0$. If $v$ lies on a
face of degree at least 8, then $v$ receives $\tfrac12$ from this face
by Rule (R7), and its new charge is then at least
$-2+2-\eps+\tfrac12\ge 0$. So we can assume that $u$ has degree 3 and
all the faces containing $v$ have degree 6. If $u$ is adjacent to some
$\ge 3$-vertex, then $u$ gives $\eps$ to $v$ by Rule (R4), and in this
case the new charge of $v$ is at least $-2+2-\eps+\eps\ge 0$. So we
can further assume that all the neighbors of $u$ are 2-vertices. Call
$u_1,u_2$ the neighbors of $u$ distinct from $v$, and for $i=1,2$ let
$v_i$ be the neighbor of $u_i$ distinct from $u$. Since $u$ has degree
3, it follows from Claim~\ref{cl:big} that $v_1$ and $v_2$ are
big. Let $v^+$ (resp. $v^-$) be the neighbor of $w$ immediately
succeeding (resp. preceding) $v$ in clockwise order around $w$. The
faces containing $v$ have degree 6, and since $G$ is bipartite with
minimum degree at least 2 (by Claims~\ref{cl:cbp} and~\ref{cl:min2}), each of
these two faces is bounded by 6 vertices. As a consequence, we can assume that $v^+$ is
adjacent to $v_1$ and $v^-$ is adjacent to $v_2$ (see Figure~\ref{fig:cont}). It follows that each
of $v^+,v^-$ has at least two big neighbors. Therefore, by Rule (R1),
$v$ received from $w$ (in addition to the
$2-\eps$ that were taken into account earlier)
$2\cdot(1-\eps)/2=1-\eps$. So the new charge of $v$ is at least
$-2+2-\eps+1-\eps=1-2\eps\ge 0$, as desired.

\begin{figure}[htbp]
\begin{center}
\includegraphics[scale=1.4]{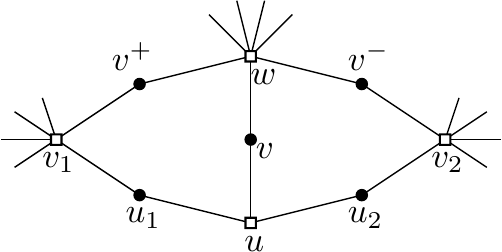}
\caption{Contact
  vertices are depicted with white squares and disk vertices are
  depicted with black dots.\label{fig:cont}}
\end{center}
\end{figure}

\smallskip

We now study the new charge of contact vertices. Note that contact
vertices give charge by Rules (R1--4) and receive charge by Rules
(R5--7). Consider a contact vertex $v$ of degree two. Then $v$ starts
with a charge of $-2$. By
Claim~\ref{cl:22}, the two neighbors of $v$ (call them $u$ an $w$)
have degree at least 3. If they both have degree at least 4, then they
both give $1+\eps$ to $v$ by Rule (R5), and the new charge of $v$ is
at least $-2+2(1+\eps)\ge \eps$. If one of $u,w$ has degree at least 4 and the other has degree 3, then $v$ receives $1+\eps$ by Rule (R5) and
$1-\eps$ by Rule (R6). In this case the new charge of $v$ is at least
$-2+1+\eps+1-\eps=0$. Finally, if $u$ and $w$ both have degree 3, then
they both give 1 to $v$ by Rule (R6), and the new charge of $v$ is at
least $-2+1+1=0$, as desired.

Consider a contact vertex $v$ of degree 3. Then $v$ starts with a
charge of 0, and only gives charge if Rule (R4) applies. In this
case, $v$ gives a charge of $\eps$ to at most two of its neighbors. However, if Rule
(R4) applies, then by definition, $v$ has a neighbor of degree at least
3. Then $v$ receives at least $1-\eps$ from such a neighbor by Rules
(R5--6). In this case, the new charge of $v$ is at least
$0-2\eps+1-\eps\ge 0$ (since $\eps= \tfrac14$).

Assume now that $v$ is a contact vertex of degree $d\ge 4$. Then $v$
starts with a charge of $2d-6$. If $v$ is small, then $v$ gives at
most $d\cdot \tfrac12$ by Rule (R3), and the new charge of $v$ is then
at least $2d-6-d\cdot \tfrac12=\tfrac{3d}2-6\ge 0$. Assume now that
$v$ is big. In this case, applications of Rule (R1) cost $v$ no more than
$d(2-\eps)$ charge. We claim
the following.

\begin{figure}[htbp]
\begin{center}
\includegraphics[scale=1.4]{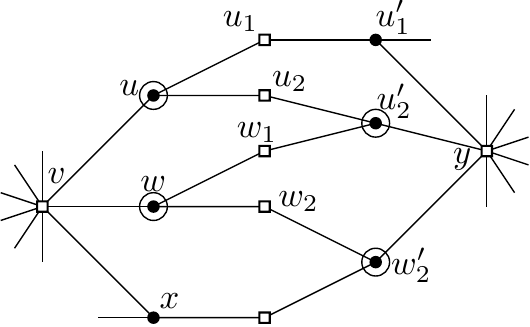}
\caption{Illustration of the proof of Claim~\ref{cl:R2}. The four removed vertices are circled.\label{fig:dis}}
\end{center}
\end{figure}

\begin{claim}\label{cl:R2}
 For every big contact vertex $v$ of degree $d$,  applications of Rule
 (R2) cost $v$ no more than $\tfrac{2d}3\cdot \eps$ charge.
\end{claim}

\begin{proof}
We will show that $v$ never gives a charge of
  $\eps$ to three consecutive neighbors of $v$, which implies the claim.
  Assume for the sake of contradiction that $v$ gives a charge of
  $\eps$ to three consecutive neighbors $u,w,x$ of $v$ (in clockwise order
  around $v$). Assume that the neighbors of $u$ are $v,u_1,u_2$ (in
  clockwise order around $u$), and the neighbors of $w$ are $v,w_1,w_2$ (in
  clockwise order around $w$). Recall that by the definition of a bad
  vertex, each of $u_1,u_2,w_1,w_2$ has degree two, and all the faces
 incident to $u$ or $w$ have degree 6. Let $u_1', u_2'$ be the
  neighbors of $u_1,u_2$ distinct from $u$, and let $w_1', w_2'$ be the
  neighbors of $w_1,w_2$ distinct from $w$. By the definition of a bad
  vertex, each of $u_1', u_2',w_1',w_2'$ has degree 3, and since all the faces
  incident to $u$ or $w$ have degree 6, $u_2'=w_1'$ and the vertices
  $u_1', u_2',w_2'$ have a common neighbor, which we call
  $y$. Again, by the definition of a bad vertex, the neighbor of
  $w_2'$ distinct from $y$ has degree two and is adjacent to $x$ (see
  Figure~\ref{fig:dis}). Let $\cF'$ be the family obtained from $\cF$
  by removing the disks corresponding to $u,w,u_2',w_2'$. By
  minimality of $\cF$, $\cF'$ has a $(k+1)$-coloring
  $c$, which we seek to extend to $u,w,u_2',w_2'$ (by a slight abuse
  of notation we identify a disk vertex of $G$ with the corresponding disk of
  $\cF$). Note that $u$
  and $w_2'$ have at most $k-1$ colored neighbors, while $w$ and $u_2'$ have at
  most $k-2$ colored neighbors. Since $k+1$ colors are available, it
  follows that each of $u,w_2'$ has a list of at least 2 available
  colors, while each of $w,u_2'$ has a list of at least 3 available
  colors. We must choose a color in each of the four lists such that
  each pair of vertices among $u,w,u_2',w_2'$, except the pair
  $uw_2'$, are assigned different colors. This is equivalent to the
  following problem: take $H$ to be the complete graph on 4 vertices minus an
  edge, assign to each vertex $z$ of $H$ an arbitrary list of at least $d_H(z)$
  colors, and then choose a color in each list such that adjacent vertices
  are assigned different colors. It follows from a classical result of
  Erd\H os, Rubin and Taylor~\cite{ERT} that this is possible for any
  2-connected graph distinct from a complete graph and an odd cycle
  (and in particular, this holds for $H$). Therefore, the $(k+1)$-coloring $c$ of $\cF'$
  can be extended to $u,w,u_2',w_2'$ to obtain a $(k+1)$-coloring of
  $\cF$, which is a contradiction. This proves Claim~\ref{cl:R2}.
\end{proof}

Hence, if $v$ is a big contact vertex of degree $d$, then the new
charge of $v$ is at least $2d-6-d(2-\eps)-\tfrac{2d}3\cdot
\eps=\tfrac{d}3\cdot \eps-6$. Since $v$ is big, $d\ge b$ and so the
new charge of $v$ is at least $b\eps/3-6= 0$ (since $b= 18/\eps$). It follows that the new charge of all vertices and
faces is nonnegative, and then the total charge (which equals $-12$)
is nonnegative, which is a contradiction. This concludes the proof of
Theorem~\ref{th:k1}.\hfill $\Box$

\section{Proof of Theorem~\ref{th:ad}}\label{sec:ad}

We start with a simple lemma showing that in order to bound the chromatic
number of $k$-touching families of Jordan curves, it is enough to
bound asymptotically the number of edges in their intersection graphs.

\begin{lemma}\label{lem:red}
Assume that there is a constant $a>0$ and a function $f=o(1)$ such that for any integers
$k,n$ and any $k$-touching family $\cF$ of $n$ Jordan curves, the graph
$G(\cF)$ has at most $ak(1+f(k))n$ edges. Then for any integer $k$, any $k$-touching family
of Jordan curves is $2ak$-colorable.
\end{lemma}

\begin{proof}
Let $\cF$ be a $k$-touching family of $n$ Jordan curves, and let $m$
denote the number of edges of $G(\cF)$.
For some integer
$\ell$, replace each element $c\in\cF$ by $\ell$ concentric copies of
$c$, without creating any new intersection point (i.e., any portion of
Jordan curve between two intersection points is replaced by $\ell$
parallel portions of Jordan curves). Let $\ell\cF$ denote the
resulting family. Note that $\ell\cF$ is $\ell k$-touching, contains
$\ell n$ elements, and $G(\ell \cF)$ contains ${\ell \choose 2} n +
\ell^2 m$ edges. Hence, we have ${\ell \choose 2} n +
\ell^2 m < a \cdot \ell k(1+f(\ell k)) \cdot \ell n$. Therefore,
$m<(ak(1+f(\ell k))-\tfrac12+\tfrac1{2\ell})n$, and $G(\cF)$ contains a vertex of
degree at most $2ak(1+f(\ell k))-1+\tfrac1{\ell}$. This holds for any $\ell$,
and since the degree of a vertex is
an integer and $f=o(1)$, $G(\cF)$ indeed contains a vertex of
degree at most $2ak-1$. We proved that $k$-touching families of
Jordan curves are $(2ak-1)$-degenerate, and therefore $2ak$-colorable.
\end{proof}

We will also need the following two lemmas.

\begin{lemma}\label{lem:faith}
For any integers $\ell,k,d$ such that $d+2\le \ell \le k$, and for any
$p \in [0,1)$,
$$(1-p)^{\ell-2}+p(1-p)^{\ell-3}(\ell-d-2) \ge (1-p)^{k-2}+p(1-p)^{k-3}(k-d-2).$$
\end{lemma}

\begin{proof}
For fixed $d \in \mathbb{R}$ and $p \in [0,1)$ we write $f(\ell):=(1-p)^{\ell-2}+p(1-p)^{\ell-3}(\ell-d-2)$. Note that $f(d+2)=(1-p)^{d}=f(d+3)$. Furthermore, for all reals $\ell \geq d+3$, $ \tfrac{d}{d\ell} f(\ell) = (1-p)^{\ell-3} \cdot \left( \log(1-p) \cdot(1+p \cdot(\ell-d-3)) +p \right) \leq  (1-p)^{\ell-3}  \left( -p \cdot(1+p \cdot(\ell-d-3)) +p \right) \leq 0.$
 So $f(\ell) \geq f(k)$ for all integers $d+2 \leq \ell \leq k$.
\end{proof}

\begin{lemma}\label{lem:expc}
For any reals $1\le \delta < 2$ and $k\ge 2$, we have
$(1-\tfrac{\delta}{k})^{k-3}\ge (1-\tfrac{\delta}{k})^{k-2}\ge e^{-\delta}$.
\end{lemma}

\begin{proof}
We clearly have $(1-\tfrac{\delta}{k})^{k-3}\ge
(1-\tfrac{\delta}{k})^{k-2}$. To see that the second part of the
inequality holds, observe first that for any real $0 \le x\le  2$, we have
$e^{-x}\le 1-\tfrac{x}3$ and thus the
desired inequality holds for $k=2,3$. 

Assume now that $k\ge 4$. Note that for any real $x\ge  0$, we have
$e^{-x}\le 1-x+\tfrac{x^2}2$. Thus,
\begin{eqnarray*}
\exp(-\tfrac{\delta}{k-2}) \le
  1-\tfrac{\delta}{k-2}+\tfrac{\delta^2}{2(k-2)^2}=1-\tfrac{\delta}{k}+\tfrac{\delta}{2k(k-2)^2}((\delta-4)k+8)\le 1-\tfrac{\delta}{k},
\end{eqnarray*}
with the rightmost inequality holding since $(4-\delta)k\ge 2k\ge 8$. It follows that $\exp(-\delta)\le
(1-\tfrac{\delta}{k})^{k-2}$, as desired.
\end{proof}

We are now ready to prove Theorem~\ref{th:ad}.

\medskip

\noindent \emph{Proof of Theorem~\ref{th:ad}.}
Let $\cF$ be a $k$-touching family of $n$ Jordan curves, with average
distance at most $\alpha k$, and let
$\delta=\delta(\alpha)$ be as provided by Theorem~\ref{th:ad}.
Note
that since $0\le \alpha \le 1$, we have $1\le \delta \le \tfrac12(1+\sqrt{5})<2$. We denote by
$E$ the edge-set of $G(\cF)$, and by $m$ the cardinality of $E$. We
will prove that $m<
\tfrac{3e^\delta\,k}{\delta+\delta^2(1-\alpha-\tfrac2k)}$. Using
Lemma~\ref{lem:red}, this implies that the chromatic number of any
$k$-touching family of Jordan curves with average distance at most
$\alpha k$ is at most
$\tfrac{6e^\delta\,k}{\delta+\delta^2(1-\alpha)}$. Note that the chosen
value $\delta(\alpha)$ of $\delta$ minimizes the value of
$\tfrac{6e^\delta}{\delta+\delta^2(1-\alpha)}$. In the remainder of the
proof, we will only use the fact that $1\le \delta < 2$.

As
observed in~\cite{EGL16}, we can assume without loss of generality
that each Jordan curve is a polygon (this is a simple
consequence of the fact that any simple plane graph can be drawn
with straight-line edges).

\smallskip

We recall that for two
intersecting Jordan curves $a,b\in \cF$, $\cD(a,b)$ is the set of
Jordan curves $c$ distinct
from $a,b$ such that the (closed) region bounded by $c$ contains
exactly one of $a,b$, and the cardinality of
$\cD(a,b)$ (which is called \emph{the distance between} $a$ and
$b$) is denoted by $d(a,b)$.
For each edge $ab\in E$, we choose an arbitrary point $x(a,b)$ in the
intersection of the Jordan curves corresponding to $a$ and
$b$. Observe that since the curves are pairwise non-crossing, $x(a,b)$
is contained in all the curves of $\cD(a,b)$.
We now select each Jordan curve of $\cF$ uniformly at random, with
probability $p=\tfrac{\delta}{k}$. Let $\cF'$ be the obtained
family. The expectation of the number of Jordan curves in $\cF'$ is
$pn$. For any pair of intersecting Jordan curves $a,b$, we denote by
$P_{ab}$ the
probability that the set $S$ of Jordan curves of $\cF'$
containing $x(a,b)$ satisfies 

\begin{enumerate}[(1)]
\item $S$ has size at most 3,
\item $a,b \in
S$, and
\item if $|S|=3$, then the Jordan curve of $S$ distinct from
$a$ and $b$ is not an element of $\cD(a,b)$. 
\end{enumerate}

Observe that 
\begin{eqnarray*}
P_{ab} & = & p^2(1-p)^{\ell-2}+p^3(1-p)^{\ell-3}(\ell-d(a,b)-2),
\end{eqnarray*}
where $\ell \in \{d(a,b)+2, \ldots,k\}$ denotes the number of Jordan curves containing $x(a,b)$
in $\cF$.

We say that an edge $ab\in E$ is \emph{good} if $a,b$ satisfy (1), (2), and (3) above.
It follows from Lemmas~\ref{lem:faith} and~\ref{lem:expc} that the expectation of the
number of good edges is

\begin{eqnarray*}
\sum_{ab \in E} P_{ab} & \ge & \sum_{ab\in E}\left(p^2(1-p)^{k-2}+p^3(1-p)^{k-3}(k-d(a,b)-2)\right)\\
& \ge & p^2 e^{-\delta} m+p^3 e^{-\delta} \sum_{ab\in E}\left( k-d(a,b)-2\right)\\
& = & p^2 e^{-\delta} m \left( 1+ p( k-2-\tfrac1{m}\sum_{ab \in E}
        d(a,b)) \right).\\
& \ge & p^2 e^{-\delta} m \left(1+\delta(1-\alpha-\tfrac2k)\right),
\end{eqnarray*}

since $\sum_{ab \in E} d(a,b)\le \alpha k m$.

\medskip

\begin{figure}[htbp]
\begin{center}
\includegraphics[scale=1.2]{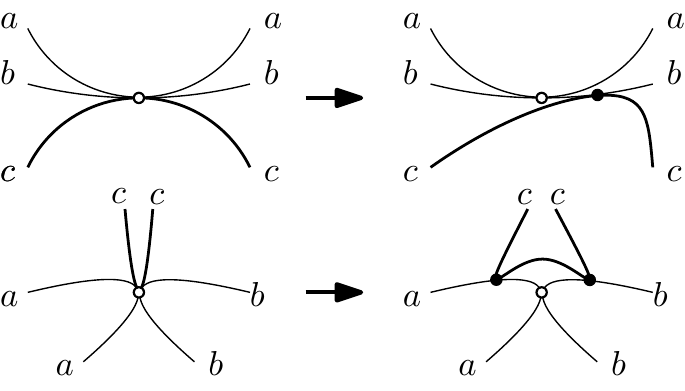}
\caption{The point $x(a,b)$ is depicted by a white dot, and the newly
  created points are depicted by black dots. \label{fig:mod}}
\end{center}
\end{figure}

Let $\cF''$ be obtained from
$\cF'$ by slightly modifying the Jordan curves around each
intersection point $x$ as follows. If $x=x(a,b)$, for some good edge
$ab$, then we do the following. Note that by (1) and (2),
$a,b\in \cF'$ and $x=x(a,b)$ is contained in at most one Jordan curve
of $\cF'$
distinct from $a$ and $b$. Assume that such a Jordan curve exists,
and call it $c$. By (3), $a$ and $b$ are at distance 0 in $\cF'$. We then slightly modify $c$ in a small disk centered
in $x(a,b)$ so that for any $d\in \{a,b\}$, if $c$ and $d$ are at distance 0
in $\cF'$, then they remain at distance 0 in $\cF''$. Moreover, the point $x(a,b)$ and the newly
created points are 2-touching in $\cF''$ (see
Figure~\ref{fig:mod}). If $ac$ (resp. $bc$) is
also a good edge with $x(a,c)=x$ (resp. $x(b,c)=x$), then note that
the conclusion above also holds with $a,b$ replaced by  $a,c$
(resp. $b,c$). 
Now, for any other intersection point $y$ of Jordan curves of $\cF'$, that is not equal to $x(a,b)$ for some good edge $ab$, we make the Jordan curves disjoint at $y$.
 It follows from the definition of a good edge that the family $\cF''$ obtained from $\cF'$ after these modifications is 2-touching, and for any good edge $ab$, $a$ and $b$ are at distance 0 in $\cF''$. Note that $G(\cF'')$ is planar, since $\cF''$ is 2-touching, and its expected number of edges is $\sum_{ab \in E} P_{ab}$. Since the number of edges of a planar graph is less than three times its number of vertices, we obtain:

\begin{eqnarray*}
3pn > \sum_{ab \in E} P_{ab}\ge p^2 e^{-\delta} m \left(1+\delta(1-\alpha-\tfrac2k)\right).
\end{eqnarray*}

As a consequence, 

$$m<\frac{3e^\delta \,k}{ \delta+\delta^2\left(1-\alpha-\tfrac2k\right)}\,n,$$

as desired. This concludes the proof of Theorem~\ref{th:ad}.\hfill
$\Box$

\section{Proof of Theorem~\ref{th:avd2}}\label{sec:avdist}

The following is an easy variation of the main result of Fox and
Pach~\cite{FP10}. Consider three Jordan curves $a,b,c$ such that $a$ is
outside the region bounded by $c$, $b$ is inside the region bounded by
$c$, and $a$ intersects $b$. Then we say that the pair $a,b$ is
\emph{$c$-crossing}.

\begin{lemma}\label{lem:fpb}
Let $c$ be a Jordan curve, and let $\cF$ be a family of $n$
Jordan curves such that $\cF\cup \{c\}$ is $k$-touching and all the
elements of $\cF$ intersect $c$. Then the number of $c$-crossing pairs
in $\cF$ is at most $2ekn$.
\end{lemma}

\begin{proof}
Let $m$ be the number of $c$-crossing pairs
in $\cF$. For each $c$-crossing pair $a,b$ in $\cF$, we consider
an arbitrary point $x(a,b)$ in $a\cap b$. We now select each Jordan curve of $\cF$ uniformly at
random with probability $p=\tfrac1{k}$. Let $\cF'$ be the resulting
family. A $c$-crossing pair $a,b$ in $\cF$ is \emph{good} if $\cF'$ contains $a$ and $b$, but does not contain any other
Jordan curve of $\cF$ containing $x(a,b)$. Note that the probability
that a given $c$-crossing pair $a,b$ is good is at least $p^2 (1-p)^{k-3}$,
and therefore the expectation of the number of good $c$-crossing pairs is
at least $p^2 (1-p)^{k-3}m$. For any intersection point $y$ of
Jordan curves of $\cF'$, that is not equal to $x(a,b)$ for some good
$c$-crossing pair $a,b$, we make the Jordan curves disjoint at $y$
(this is possible since the Jordan curves are pairwise
non-crossing). Let $\cF''$ be the obtained family. Observe that $\cF''$ is
2-touching and each intersection point contains one Jordan curve
lying outside the region
bounded by $c$ and one Jordan curve lying inside the region
bounded by $c$. The graph $G(\cF'')$ is therefore planar and
bipartite. The expectation of the number of vertices of $G(\cF'')$ is
$pn$ and the expectation of the number of edges of $G(\cF'')$ is at
least $p^2 (1-p)^{k-3}m$. Since any planar bipartite graph on $N$
vertices contains at most $2N$ edges, it follows that $p^2(1-
p)^{k-3}m < 2pn$. Since $(1-\tfrac1{k})^{k-3}>e^{-1}$, we obtain
that $m<2ekn$, as desired.
\end{proof}

Some planar quadrangulations can be represented as 2-touching families
of Jordan curves intersecting a given Jordan curve $c$ (so that each
edge of the quadrangulation corresponds to a $c$-crossing pair of
Jordan curves). Therefore, the bound $2N$ cannot be decreased (by more
than an additive constant) in the
proof of Lemma~\ref{lem:fpb}. Furthermore, the possibly near-extremal example in Figure \ref{fig:counterexample} shows that the bound $2ekn$ in Lemma~\ref{lem:fpb} cannot be improved to less than $(2k-4)n$.

\begin{figure}[htbp]
\begin{center}
\includegraphics[scale=0.3]{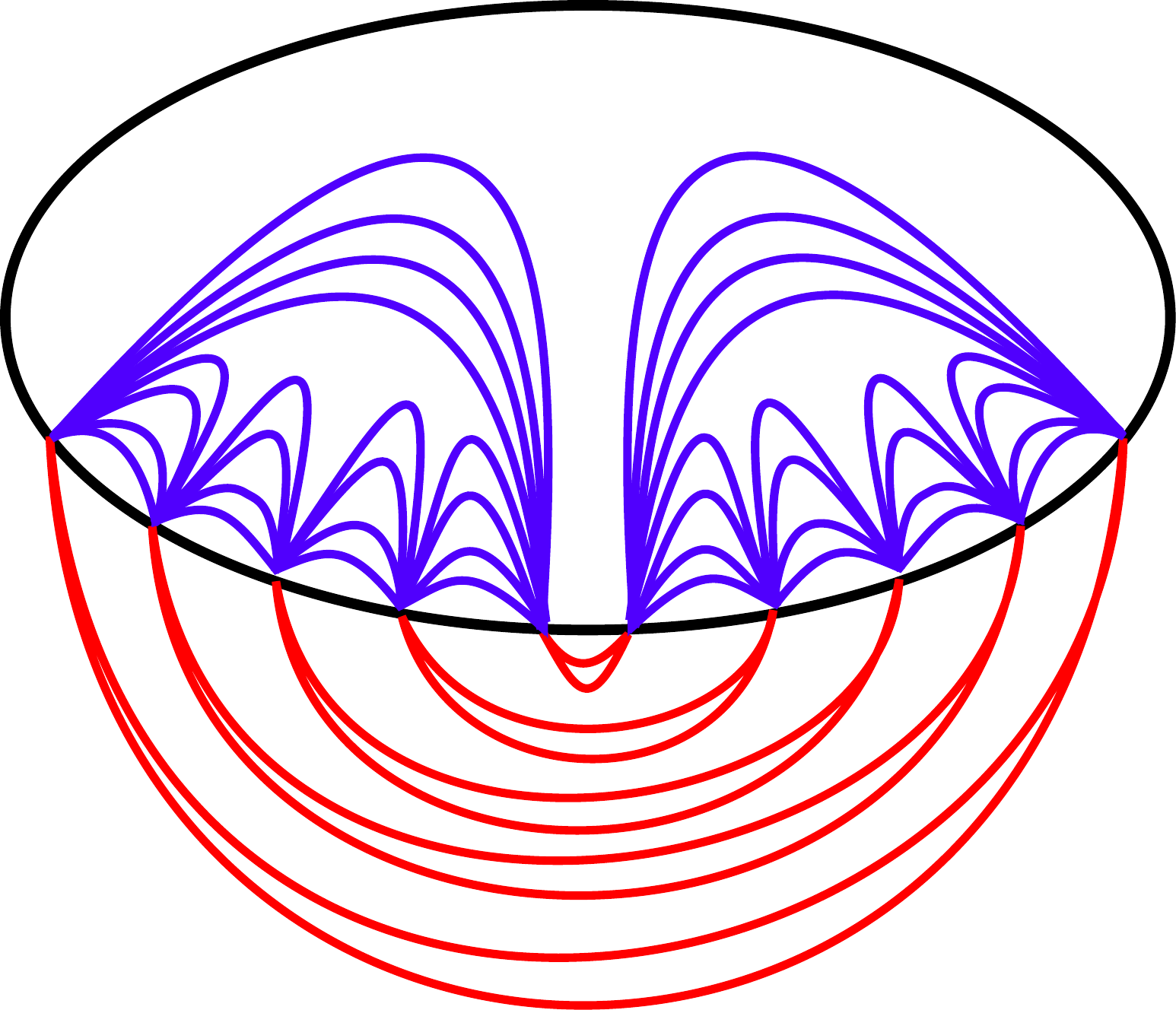}
\caption{A family $\cF$ of $n$ Jordan curves that all intersect a
  fixed Jordan curve $c$, such that $\cF \cup \left\{c\right\}$ is
  $k-$touching. Each Jordan curve in the interior of $c$ touches each
  Jordan curve in the exterior of $c$. In the interior, there are
  only two sets of $k-2$ concentric Jordan curves. The remaining
  $n-2k+4$ Jordan curves are in the exterior of $c$. As $n$ goes to infinity, the number of $c-$crossing pairs divided by $n$ converges to $2k-4$. \label{fig:counterexample}}
\end{center}
\end{figure}

\smallskip

We are now ready to prove Theorem~\ref{th:avd2}.

\medskip

\noindent \emph{Proof of Theorem~\ref{th:avd2}.}
Let $E$ denote the edge-set of $G(\cF)$ and let $m=|E|$. Let $\alpha=\tfrac1{km}\sum_{ab\in
  E} d(a,b)$. Note that the average distance in $\cF$ is $\alpha k$. 

Fix some $0<\epsilon<1$, and set $\delta=\tfrac12(1-\epsilon)$. For any edge
$ab\in E$ with $d(a,b)>0$, we do the following. Note that there is
a unique ordering $c_1,\ldots,c_d$ of the elements of $\cD(a,b)$,
such that for any $1\le i \le d$, the distance between $a$ and $c_i$
is $i-1$. Then the edge $ab$ gives a charge of 1 to each of the
elements $c_{\lceil \delta d\rceil},c_{\lceil \delta d\rceil+1},\ldots,c_{\lfloor (1-\delta)d\rfloor+1}$. Let $T$ be the total
charge given during this process. Note that

\begin{eqnarray*}
 T& = & \sum_{ab\in
  E}(\lfloor (1-\delta)d(a,b)\rfloor+1-\lceil \delta d(a,b)\rceil+1)\\
& \ge & \sum_{ab\in
  E} (1-2\delta) d(a,b)=\epsilon\sum_{ab\in
  E} d(a,b).
\end{eqnarray*}

We now analyze how much charge was received by an arbitrary
Jordan curve $c$. Let $N(c)$ denote the neighborhood of $c$, and let
$N^+(c)$ (resp. $N^-(c)$)
denote the set of neighbors of $c$ lying outside (resp. inside) the region bounded by
$c$. Observe that if $c$ received a charge of 1 from some edge $ab$,
then without loss of generality we have $a\in N^+(c)$, $b\in N^-(c)$,
and both $a$ and $b$ are at distance at most $\max(\lfloor
(1-\delta)d(a,b)\rfloor,d(a,b)-\lceil \delta d(a,b)\rceil)\le
(1-\delta) d(a,b)\le (1-\delta)k$ from $c$.
Let $N_{1-\delta}(c)$ denote the set of neighbors of $c$ that are at
distance at most $(1-\delta)k$ from $c$. Then the charge received by
$c$ is at most the number of $c$-crossing pairs $a,b$ in the subfamily of $\cF$ induced
by $N_{1-\delta}(c)$, which is at most $2e k|N_{1-\delta}(c)|$ by Lemma~\ref{lem:fpb}.

For any $\gamma$, let $m_{\gamma}$ denote the number of edges $ab\in E$ such that $a$ and
$b$ are at distance at most  $\gamma k$. It follows from the
analysis above that $T\le 4ekm_{1-\delta}$. Therefore, $\sum_{ab\in
  E} d(a,b)\le \tfrac{4e}{\epsilon}km_{1-\delta} $. Since $\sum_{ab\in
  E} d(a,b)=\alpha k m$, we have $m_{(1+\epsilon)/2}=m_{1-\delta}\ge
\tfrac{\epsilon \alpha}{4e} \,m$. 

\smallskip

We now
study the contribution of an arbitrary edge $ab$ to the sum $\sum_{ab\in
  E} d(a,b)=\alpha k m$. Let $t$ be some integer. If $d(a,b)\le \tfrac{t+1}{2t} k$, then
$ab$ contributes at most $\tfrac{t+1}{2t} k$ to $\alpha k m$,
and therefore at most $\tfrac{t+1}{2t} $ to $\alpha  m$. Note
that there are $m_{(t+1)/2t}$ such edges $ab$. For
each $2\le i \le t-1$, each edge $ab$ such that
$\tfrac{t+i-1}{2t} k< d(a,b)\le \tfrac{t+i}{2t} k$
contributes at most $\tfrac{t+i}{2t} $ to $\alpha  m$, and there
are $m_{(t+i)/2t}-m_{(t+i-1)/2t}$ such
edges. Finally, each edge $ab$ with $d(a,b)> \tfrac{2t-1}{2t} k$
contributes at most 1 to $\alpha  m$, and there are
$m-m_{(2t-1)/2t}$ such edges. As a consequence, 

\begin{eqnarray*}
\alpha m & \le & \tfrac{t+1}{2t} m_{(t+1)/2t}+\sum_{i=2}^{t-1} \left(
                 \tfrac{t+i}{2t} (m_{(t+i)/2t}-m_{(t+i-1)/2t}
                 )\right)+ m-m_{(2t-1)/2t}\\
& = & \sum_{i=1}^{t-1} \left(m_{(t+i)/2t} (\tfrac{t+i}{2t}-\tfrac{t+i+1}{2t})
                 \right)+m\\
& = & m-\tfrac1{2t} \sum_{i=1}^{t-1} m_{(t+i)/2t}\\
& \le & m - \tfrac1{2t} \sum_{i=1}^{t-1}\tfrac{i}{t} \tfrac{\alpha}{4e} \,m ,
\end{eqnarray*}

since $m_{(1+\epsilon)/2}\ge
\tfrac{\epsilon \alpha}{4e} \,m$ for every $0<\epsilon<1$. As a
consequence, we obtain that $\alpha \le 1-\tfrac{t-1}t
\tfrac{\alpha}{16e}$. Since this holds for any integer $t$, we have
$\alpha \le 1-\tfrac{\alpha}{16e}$ and therefore $\alpha\le
1/(1+\tfrac1{16e})$, as desired.\hfill $\Box$

\section{Remarks and open questions}\label{sec:ccl}

Most of the proof of Theorem~\ref{th:k1} proceeds by finding a
Jordan region intersecting at most $k$ other Jordan regions (see
Claim~\ref{cl:k}). On a single occasion, we use a different reduction (via
a list-coloring argument). A natural question is: could this be
avoided? Is it true that in any simple $k$-touching family of
Jordan regions, if $k$ is large enough, then there is a
Jordan region which intersects at most $k$ other Jordan regions? It turns
out to be wrong, as depicted in Figure~\ref{fig:hands}. However, a proof along the
lines of that of Theorem~\ref{th:k1} (but significantly simpler),
shows that if $k$ is large enough, then there is a
Jordan region which intersects at most $k+1$ other Jordan
regions. It was pointed out to us by Patrice
Ossona de Mendez (after the original version of this manuscript was submitted) that he also obtained this result in 1999
(see~\cite{OM99}). His result and its proof are stated with a completely different
terminology, but the ideas are essentially the same. In particular,
his result also implies (relatives of) our Corollaries~\ref{cor:1} and~\ref{cor:2}.

\begin{figure}[htbp]
\begin{center}
\includegraphics[scale=1.2]{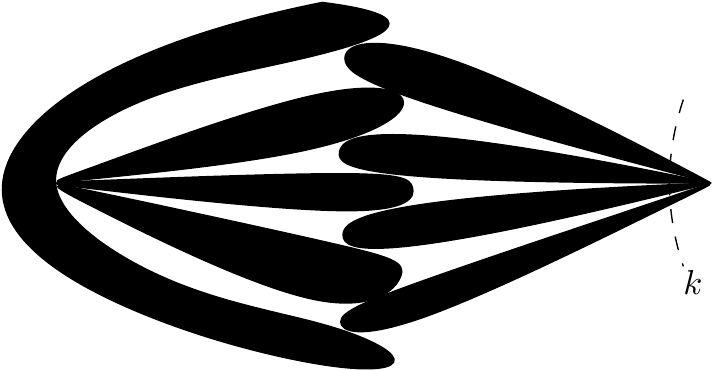}
\caption{Every Jordan region intersects precisely $k+1$ other Jordan regions. \label{fig:hands}}
\end{center}
\end{figure}

This can be used to obtain a result on the chromatic number of \emph{simple}
families of $k$-touching Jordan curves (families of $k$-touching
Jordan curves such that any two Jordan curves intersect in at most
one point). Using the result mentioned above (that if $k$ is
sufficiently large and the interiors are pairwise disjoint, then there
is a Jordan region that intersects at most $k+1$ other Jordan regions), it
is not difficult to show that the chromatic number of any simple
family of $k$-touching Jordan curves is at most $2k$ plus a constant. We believe that the answer should be much smaller.

\begin{prob}\label{pr:31}
Is it true that for some
constant $c$, any simple family of $k$-touching Jordan curves can
be colored with at most $k+c$ colors?
\end{prob}

It was conjectured in~\cite{EGL16} that if $\cS$ is a family of pairwise
non-crossing strings such that (i) any two strings intersect in at
most one point and (ii) any point of the plane is on at most $k$
strings, then $\cS$ is $(k+c)$-colorable, for some constant $c$. Note
that, if true, this conjecture would give a positive answer to
Problem~\ref{pr:31}.

\end{document}